\documentclass{amsart}
\usepackage{hyperref}
\hypersetup{nesting=true,debug=true,naturalnames=true}
\usepackage{graphicx,amssymb,upref}

\usepackage{amsmath}
\usepackage{amsthm}
\usepackage{multirow}
\usepackage{array}
\usepackage[english]{babel}
\usepackage{lmodern}
\usepackage{microtype}
\usepackage[utf8]{inputenc}
\usepackage[T1]{fontenc}
\usepackage{mathtools}

\newtheorem{theorem}{Theorem}[section]
\newtheorem{lemma}[theorem]{Lemma}
\newtheorem{corollary}[theorem]{Corollary}

\theoremstyle{definition}
\newtheorem{definition}[theorem]{Definition}
\newtheorem{proposition}[theorem]{Proposition}
\newtheorem{example}[theorem]{Example}

\theoremstyle{remark}
\newtheorem{remark}[theorem]{Remark}

\numberwithin{equation}{section}

\author{Bartosz Naskr\k{e}cki}
\address{Faculty of Mathematics and Computer Science, Adam Mickiewicz University\\
Umultowska 87, 61-614 Poznań, Poland\\
and School of Mathematics, University of Bristol, University Walk, Bristol BS8 1TW, UK}
\email{nasqret@gmail.com} 
\title[Divisibility sequences of polynomials]{Divisibility sequences of polynomials and heights estimates}

\thanks{The author was supported by the National Science Centre Poland research grant 2012/07/B/ST1/03541 and by the DFG grant Sto299/11-1 within the framework of the Priority Programme SPP 1489.}

\newcommand{\summy}{\sum_{\substack{ %
   m\mid n \\
   m<n %
   }}}
   
\newcommand{\heightE}[1]{\widehat{h}_{E}(#1)}
\newcommand{\pairing}[2]{\left\langle #1,#2\right\rangle}
\newcommand{\dotprod}[2]{\overline{#1}.\overline{#2}}
\DeclareMathOperator{\Div}{Div}
\DeclareMathOperator{\charac}{char}
\DeclareMathOperator{\ord}{ord}
\DeclareMathOperator{\GCD}{GCD}
\DeclareMathOperator{\LCM}{LCM}

\DeclareMathOperator{\Supp}{Supp}
\DeclareMathOperator{\rank}{rank}

\begin{document}

\begin{abstract}
In this note we compute a constant $N$ that bounds the number of non--primitive divisors in elliptic divisibility sequences over function fields of any characteristic. We improve a result of
Ingram--Mah{\'e}--Silverman--Stange--Streng, 2012, and we show that the constant can be chosen independently of the specific point and to some extent of the specific curve, as predicted in loc. cit.
\end{abstract}

\maketitle

\section{Introduction}
Let $E$ be an elliptic curve over the function field $K(C)$ of a smooth projective curve $C$ of genus $g(C)$ over an algebraically closed field $K$. Let $S$ be the Kodaira--N\'{e}ron model of $E$, i.e. a smooth projective surface with a relatively minimal elliptic fibration $\pi:S\rightarrow C$ with a generic fibre $E$ and a section $O:C\rightarrow S$, cf. \cite[\S 1]{Shioda_Mordell_Weil}, \cite[Chap. III, \S 3]{Silverman_book}. We always assume that $\pi$ is not smooth. Let $P$ be a point of infinite order in the Mordell--Weil group $E(K(C))$. To formulate the main problem we define a family of effective divisors $D_{nP}\in\Div(C)$ parametrized by natural numbers $n$. For each $n\in\mathbb{N}$ the divisor $D_{nP}$ is the pullback of the image $\overline{O}$ of section $O$ through the morphism $\sigma_{nP}:C\rightarrow S$ induced by the point $nP$
\[D_{nP}=\sigma_{nP}^{*}(\overline{O}).\]
We call such a family an \emph{elliptic divisibility sequence}. We say that the divisor $D_{nP}$ is \textit{primitive} if the support of $D_{nP}$ is \textit{not} completely contained in the sum of supports of the divisors $D_{mP}$ for all $m<n$. Otherwise we say that the divisor $D_{nP}$ is \emph{non--primitive}.

The study of elliptic divisibility sequences dates back to the work of Morgan Ward \cite{Ward_The_law, Ward_EDS_Memoir}. Silverman in \cite{Silverman_Wieferich} established that for elliptic divisibility sequences over $\mathbb{Q}$ the number of non--primitive divisors is finite. This result was investigated further by several authors \cite{Everest_Ingram_Stevens_Primitive, Everest_Mclaren_Ward_EDS, Ingram_EDS_over_curves, Silverman_Ingram_uniform, Kuhn_Muller, Stange_Elliptic_nets}. In another direction Streng \cite{Streng_Divisibility_CM} generalized the primitive divisor theorems for curves with complex multiplication. Several authors studied also the question of existence of perfect powers in divisibility sequences, cf. \cite{Everest_Ingram_Mahe_Stevens_Uniform_EDS, Everest_Reynolds_On_the_denominators, Reynolds_Perfect}. In the context of elliptic divisibility sequences over function fields the finiteness of the set of non--primitive divisors for elliptic curves over $\mathbb{Q}(t)$ was proved in \cite{Everest_Ingram_Mahe_Stevens_Uniform_EDS}. In parallel such questions have been studied also for Lucas sequences \cite{Flatter_Ward}. In \cite{Silverman_Common_divisors} common divisors of two distinct elliptic divisibility sequences were studied.  For a general function field of a smooth curve in characteristic zero, the first general theorem about primitive divisors in elliptic divisibility sequences was proved in \cite{Silverman_divisibility}. The authors of \cite{Silverman_divisibility} ask the following question: \emph{For a fixed elliptic curve $E$ over a function field and a point $P$ of infinite order is it possible to give an explicit upper bound for the value of a constant $N=N(E,P)$ such that for all $n\geq N$ the divisor $D_{nP}$ in the elliptic divisibility sequence is primitive?}

Such a bound $N(E,P)$ always exists by \cite[Thm. 5.5]{Silverman_divisibility} but the proof does not indicate how to make the bound explicit or uniform with respect to $E$ and $P$.

In this note we investigate the existence of uniform bounds for the number of non--primitive divisors. In Section \ref{sec:main_theorems} we formulate our main theorems. There is a considerable difference between the formulation and proof of theorems in characteristic zero and positive so we do state them separately. In Section \ref{sec:notation} we establish necessary notation that will be used through the paper. In Section \ref{sec:prelimin} we gather basic facts about the canonical height function and the relation between the discriminant divisor of an elliptic curve and the Euler characteristic of the attached elliptic surface. The crux is the explicit recipe for the height function due to Shioda \cite{Shioda_Mordell_Weil}, that will be used in critical places to get the estimate on the number of non--primitive divisors in the divisibility sequence. Section \ref{sec:arithm} contains a couple of properties of arithmetic functions used in the proofs of main theorems. In Section \ref{sec:bounds_heights} we discuss the analogue of Lang's conjecture on canonical height of points over function fields. We use the results of \cite{SilvHindry} and \cite{Pesenti_Szpiro} to produce effective bounds for fields of arbitrary characteristic.

In Section \ref{sec:char_0_section} we explain a relatively simple proof of theorems formulated for function fields of characteristic $0$. The main idea of the proof is to combine the explicit approach to height computations of \cite{Shioda_Mordell_Weil} with the bounds for minimal heights of points proved in \cite{SilvHindry}. A crucial step in the proof relies on the formula that relates the Euler characteristic $\chi(S)$ to the sum of numbers that depended on the Kodaira types of singular fibres of $\pi$.

In Section \ref{sec:positive_char} we prove the main theorems in positive characteristic. The main steps of the proof are similar to the characteristic $0$ case, however there are significant differences due to the presence of inseparable multiplication by $p$ map. In the last section we gather several examples for which we compute explicitly the exact number of non--primitive divisors. We also explain how the main theorems fail in positive characteristic $p$ for elliptic curves with $p$--map of inseparable degree $p^2$.

\section{Main theorems}\label{sec:main_theorems}
Our convention is to work with function fields $K(C)$ over algebraically closed field $K$ of constants. However, the main theorems can be formulated for a smooth, projective geometrically irreducible curve $C$ over a field $K$ that is a number field or a finite field. In such a case, an elliptic curve $E$ is defined over the field $K(C)$ and the elliptic surface $\pi:S\rightarrow C$ attached to $E/K(C)$ is a regular scheme $S$ over $K$ with a proper flat morphism $\pi$ into $C$ and such that its base change to the algebraic closure $\overline{K}$ is an elliptic surface in the usual sense. Every point $v\in C(\overline{K})$ corresponds to a normalized valuation of $\overline{K}(C)$. We say that $v$ is a \emph{primitive valuation} of $D_{nP}$ when $v$ is contained in the support of $D_{nP}$ and does not belong to the support of any $D_{mP}$ for $m<n$, cf. \cite[Def. 5.4]{Silverman_divisibility}. In this terminology we can say that $D_{nP}$ is primitive if and only if it has a primitive valuation and similarly $D_{nP}$ is non--primitive whenever it does not have a primitive valuation.

From now on we assume that $K=\overline{K}$, unless otherwise specified. Let $E$ be an elliptic curve over the field $K(C)$ with at least one fibre of bad reduction and let $P$ be a point of infinite order in $E(K(C))$. Let $\pi:S\rightarrow C$ be an elliptic surface attached to $E$. Consider a divisibility sequence $\{D_{nP}\}_{n\in\mathbb{N}}$. 

\begin{theorem}\label{thm:dep_on_g}
Let $K(C)$ be a field of characteristic $0$. There exists a constant $N=N(g(C))$ which depends only on the genus of $C$, such that for all $n\geq N$ the divisor $D_{nP}$ has a primitive valuation.
\end{theorem}

\begin{theorem}\label{thm:dep_on_chi}
Let $K(C)$ be a field of characteristic $0$. There exists a constant $N=N(\chi(S))$ which depends only on the Euler characteristic of surface $S$, such that for all $n\geq N$ the divisor $D_{nP}$ has a primitive valuation.
\end{theorem}
Proofs of both theorems are presented in Section \ref{sec:char_0_section}. 

Now let us assume that $p=\charac K(C)\geq 5$. Let $p^{r}$ be the inseparable degree of the $j$--map of $E$ if $j$ is non--constant, otherwise we put $1$. Let us assume that the multiplication by $p$--map has inseparable degree $p$. We say that $E$ is tame when locally at all places the valuation of the leading term of the formal group homomorphisms $\widehat{[p]}$ is less than $p$. Otherwise we say that $E$ is wild, cf. Definition \ref{def:tame_wild_def}. Both assumptions imply that $E$ is ordinary or in other words that it has ordinary reduction at all places, cf. Section \ref{sec:positive_char}.

\begin{theorem}[Theorem \ref{thm:main_theorem_tame}]
Assume that $E$ is ordinary and tame. There exists an explicit constant $N=N(g(C),p,r)$ which depends only on the genus of $C$, $p$ and $r$ such that for all $n\geq N$ the divisor $D_{nP}$ has a primitive valuation.
\end{theorem}
\begin{theorem}[Theorem \ref{thm:theorem_wild_case}]
Let $E$ be an elliptic curve defined over $K(C)$ of characteristic $p>3$ with field of constants $K=\mathbb{F}_{q}$, $q=p^{s}$. Let $E$ be ordinary and wild. There exists an explicit constant $N=N(g(C),\chi(S),p,r,s)$ which depends only on the genus of $C$, Euler characteristic $\chi(S)$, $p$, $r$ and $s$ such that for all $n\geq N$ the divisor $D_{nP}$ has a primitive valuation.
\end{theorem}
When the multiplication by $p$ map is of inseparable degree $p^2$ we can find examples of curves with infinitely many non--primitive divisors in the divisibility sequence. They are discussed in Section \ref{sec:Examples}.

\section{Notation}\label{sec:notation}
\begin{itemize}
 \item $\chi(S)$ -- the Euler characteristic $\chi(S,\mathcal{O}_{S})$ of a surface $S$
 \item $g(C)$ -- the genus of a curve $C$
 \item $K(C)$ -- the function field of a curve $C$ over a field of constants $K$; the field $K$ will usually be algebraically closed, unless otherwise specified
 \item $E$ -- an elliptic curve over $K(C)$
 \item $j$ -- the $j$-invariant of $E$
 \item $\Delta_{E}$ -- the minimal discriminant divisor of $E$
 \item $\heightE{P}$ -- the canonical height of a point $P$
 \item $h_{K(C)}(E)$ -- the height of $E$ defined to be $h_{K(C)}(E)=\frac{1}{12}\deg\Delta_{E}$
 \item $\{D_{nP}\}_{n\in\mathbb{N}}$ -- a divisibility sequence attached to a point $P$
\end{itemize}

\section{Preliminaries}\label{sec:prelimin}
We will use the notation similar to that in \cite{Shioda_Mordell_Weil}. By $\langle\cdot,\cdot\rangle:E(K(C))\times E(K(C))\rightarrow \mathbb{Q}$ we denote the symmetric bilinear pairing on $E(K(C))$ which induces the structure of a positive--defined lattice on $E(K(C))/E(K(C))_{\textrm{tors}}$, cf. \cite[Thm. 8.4]{Shioda_Mordell_Weil}. The pairing $\langle\cdot,\cdot\rangle$ induces the height function $P\mapsto \langle P,P\rangle$ which corresponds to the canonical height. For a point $P\in E(K(C))$ we denote by $\overline{P}$ the image of its associated section $\sigma_{P}:C\rightarrow S$ in the given elliptic surface model. By $C_{1}.C_{2}$ we denote the intersection pairing of two curves $C_{1}$, $C_{2}$ lying on $S$. We denote by $G(F_{v})$ the group of simple components of the fibre $F_{v}=\pi^{-1}(v)$ above $v\in C$. In Figure \ref{fig:group_components}, following \cite[Chap. IV, \S 9]{Silverman_book}, we present all possible group structures of $G(F_{v})$ corresponding to different Kodaira types of singular fibres $F_{v}$. We denote by $B$ the set of all places $v\in C$ of bad reduction.
\begin{figure}[htb]
\begin{align*}
G(I_{n})&\cong \mathbb{Z}/n\\
G(I_{2m}^{*})&\cong (\mathbb{Z}/2)^2\\
G(I_{2m+1}^{*})&\cong (\mathbb{Z}/4)\\
G(II)\cong G(II^{*})&\cong \{0\}\\
G(III)\cong G(III^{*})&\cong \mathbb{Z}/2\\
G(IV)\cong G(IV^{*})&\cong \mathbb{Z}/3
\end{align*}
\caption{Group of components of fibre with a certain Kodaira type}\label{fig:group_components}
\end{figure}

\begin{figure}[htb]
$\begin{array}{c|cccccc}
  \textrm{type of }F_{v} & III & III^{*} & IV & IV^{*} & I_{b}\ (b\geq 2) & I_{b}^{*}\ (b\geq 0)\\
  \hline
 \begin{array}{c}
 c_{v}(P),\\
 i=\mathop{comp}_{v}(P)
 \end{array} & 1/2 & 3/2 & 2/3 & 4/3 & i(b-i)/b & \left\{\begin{array}{ll}1 & (i=1)\\ 1+b/4 & (i>1)\end{array}\right.\\
 \hline
 \begin{array}{c}
  c_{v}(P,Q),\\
  i=\mathop{comp}_{v}(P),\\
  j=\mathop{comp}_{v}(Q),\\
  i<j
 \end{array} & - & - & 1/3 & 2/3 & i(b-j)/b & \left\{\begin{array}{ll}1/2 & (i=1)\\ 2+b/4 & (i>1)\end{array}\right.
\end{array}$
\caption{Values of correcting terms $c_{v}(P,Q)$ for all possible singular fibre types with at least two components}\label{fig:corrections}
\end{figure}
By \cite[(2.31)]{Shioda_Mordell_Weil} it is possible to write the height pairing in terms of explicit numbers. We denote by $c_{v}(P,Q)$ the correcting terms that are determined by computation of intersection of curves $\overline{P}$ and $\overline{Q}$ in the fibre above $v$, cf. Figure \ref{fig:corrections} reproduced from \cite[8.16]{Shioda_Mordell_Weil}. The values $c_{v}(P,Q)$ depend on the numbering of components in the fibre above $v$. For a point $P$ we denote by $comp_{v}(P)$ the component above $v$ that intersects the curve $\overline{P}$. For a fibre $F_{v}$ above $v$ we only label the simple components. The unique component that intersects the image of the zero section $\overline{O}$ is denoted by $\Theta_{v,0}$ and we put $comp_{v}(P)=0$ if the image $\overline{P}$ intersects $\Theta_{v,0}$. For the fibres of type $I_{n}$ with $n>1$ we put labels $\Theta_{v,0}, \Theta_{v,1}$, \ldots, $\Theta_{v,n-1}$ cyclically, fixing one of two possible choices. For $F_{v}$ of type $I_{n}^{*}$ we denote by $\Theta_{v,1}$ the component which intersects the same double component as $\Theta_{v,0}$. The other two simple components $\Theta_{v,2}$ and $\Theta_{v,3}$ are labelled in an arbitrary way. For the other additive reduction types we choose one fixed labelling (the order is irrelevant). For two points $P$ and $Q$ we put $c_{v}(P,Q)=0$ whenever $comp_{v}(P)=0$ or $comp_{v}(Q)=0$. The non--trivial cases are described in Figure \ref{fig:corrections}.  In \cite[Thm. 8.6]{Shioda_Mordell_Weil} it is proved that
\[\left\langle P,Q \right\rangle = \chi(S)+\overline{P}.\overline{O}+\overline{Q}.\overline{O}-\overline{P}.\overline{Q}-\sum_{v\in B}c_{v}(P,Q).\]
In particular we have the equality
\begin{equation}\label{eq:Shioda_height}
\pairing{P}{P} = 2\chi(S)+2\dotprod{P}{O}-\sum_{v\in B}c_{v}(P,P)
\end{equation}
The notion of canonical height from \cite[\S 1]{SilvHindry} is slightly different from the notion of the height determined by $\langle\cdot,\cdot\rangle$. In fact the first is defined by the limit
\[\widehat{h}_{E}(P)=\lim_{n\rightarrow\infty}\frac{\deg \sigma^{*}_{nP}\overline{O}}{n^2}.\]
using our notation. By \cite[Chap. III Thm. 9.3]{Silverman_book} the following equality holds
\begin{equation}\label{eq:height_def_comparision}
\widehat{h}_{E}(P)=\frac{1}{2}\langle P,P\rangle.
\end{equation}
\noindent
We also remark that $\deg \sigma^{*}_{nP}\overline{O}=\deg D_{nP}=\overline{nP}.\overline{O}$ which clearly follows from the definition.

For a fibre above $v$ let us denote by $m_{v}$ the number of irreducible components in $F_{v}$.
For the fibre $F_{v}=\pi^{-1}(v)$ with $m_{v}$ components the Euler number $e(F_{v})$ (cf. \cite[Prop. 5.1.6]{Dolgachev}) equals $0$ at $v$ of good reduction, $m_{v}$ at places $v$ of bad multiplicative reduction and $m_{v}+1$ at places of bad additive reduction. 
\[e(F_{v}) = \left\{
  \begin{array}{ll}
    0 & v\textrm{ has good reduction}\\
    m_{v} & v\textrm{ has multiplicative reduction}\\
    m_{v}+1 & v\textrm{ has additive reduction}
  \end{array}
\right.\]
By \cite[Thm. 2.8]{Shioda_Mordell_Weil} it follows that the square $K_{S}^{2}$ of the canonical bundle $K_{S}$ is $0$ and by Noether's formula \cite[Chap. V, Rem. 1.6.1]{Hartshorne} and \cite[Prop. 5.1.6]{Dolgachev}
\begin{equation}\label{eq:euler_formula}
12\chi(S)=e(S)=\sum_{v\in B}(e(F_{v})+\delta_{v}).
\end{equation}
The terms $\delta_{v}$ are non--negative and non-zero only in the special cases of $\textrm{char}\:K=2,3$.

We denote by $\Delta_{E}$ the sum $\sum_{v\in C}(\ord_{v}\Delta_{v}) \: (v)$ where $\ord_{v}\Delta_{v}$ is the order of vanishing of the minimal discriminant $\Delta_{v}$ of $E$ at $v$.
On the other hand by Tate's algorithm \cite{Tate_algorithm_article} $e(F_{v})$ equals $\ord_{v}\Delta_{v}$  when characteristic $p$ equals $0$ or is greater than $3$. This implies the equalities
\[h_{K(C)}(E)=\frac{1}{12}\deg\Delta_{E}=\frac{1}{12}\sum_{v\in C}(\ord_{v}\Delta_{v})\:  (v)=\frac{1}{12}e(S)=\chi(S).\]

\section{Arithmetic functions}\label{sec:arithm}
We will use further two arithmetical functions: 
\begin{align*}
d(n)&=\sum_{m\mid n} 1,\\
\sigma_{2}(n)&=\sum_{m \mid n} m^2.
\end{align*}
For the applications in Section \ref{sec:char_0_section} it is often enough to use the trivial bound $d(n)\leq n$. However, for the applications in Section \ref{sec:positive_char} a stronger bound \cite{Nicolas_Robin} is required
\begin{equation}\label{eq:d_bound}
d(n)\leq n^{1.5379\log 2/\log\log n}\quad\textrm{ for }n\geq 3.
\end{equation}
We easily obtain the following estimate
\begin{align*}
\sigma_{2}(n)&=\sum_{m\mid n} m^2 = n^2\prod_{p^{\alpha}\mid\mid n} (1+p^{-2}+\ldots+p^{-2\alpha})\\
&\leq n^2 \prod_{p\mid n} (1+p^{-2}+\ldots)=n^2\prod_{p\mid n}\left(\frac{1}{1-p^{-2}}\right)\\
&\leq n^2\prod_{p}\left(\frac{1}{1-p^{-2}}\right)=n^2\zeta(2)<n^2\cdot 1.645
\end{align*}
It implies that for any $n>0$ we have
\begin{equation}\label{eq:sigma_2_estimate}
\sigma_{2}(n)< \zeta(2)n^2< 1.645n^2.
\end{equation}
For a fixed prime number $p$ we define also functions 
\begin{align*}
d^{(p)}(n)&=\sum_{m\mid n}p^{v_{p}(n/m)},\\
\sigma_{2}^{(p)}(n)&=\sum_{m\mid n}p^{v_{p}(n/m)} m^2.
\end{align*}
We denote by $v_{p}(n)$ the standard $p$--adic valuation of $n$ at $p$.
\begin{proposition}\label{prop:arithm_estimates}
The functions $\sigma_{2}^{(p)}(n)$ and $d^{(p)}(n)$ are multiplicative and they satisfy:
\begin{itemize}
 \item $d^{(p)}(n)=\frac{p^{e+1}-1}{(e+1)(p-1)}\cdot d(n)$
 \item $\sigma_{2}^{(p)}(n)=\frac{p^{e}(p+1)}{p^{e+1}+1}\sigma_{2}(n)<(1+\frac{1}{p})\zeta(2)n^2$
\end{itemize}
where $n=n_{0}p^{e}$, $p\nmid n_{0}$ and $e=v_{p}(n)$.
\end{proposition}
\begin{proof}
Put $f(n)=p^{v_{p}(n)}$. We observe that $d^{(p)}(n)$ is the Dirichlet convolution of $d(n)$ with $f(n)$. Similarly $\sigma_{2}^{(p)}(n)$ is a convolution of $f(n)$ with $\sigma_{2}(n)$. The multiplicativity follows and the rest is an easy exercise.
\end{proof}

\section{Bounds on the canonical height}\label{sec:bounds_heights}
In this section we collect together certain lower bounds on canonical height $\heightE{P}$ of a point of infinite order. The first presented bound is slightly weaker than the analogue of Lang's conjecture \cite{SilvHindry} but its proof relies entirely on the theory of Mordell--Weil lattices and the outcome does not depend on the characteristic of the field $K(C)$.
\begin{lemma}\label{lemma:inequality3}
Assume $E$ is an elliptic curve over $K(C)$. Let $P$ be a point of infinite order in $E(K(C))$. Then 
\begin{equation*}
1/\heightE{P}\leq 24\cdot 3^{4\chi(S)}.
\end{equation*}
\end{lemma}
\begin{proof}
If $P$ is a point of infinite order in $E(K(C))$, then the height $\langle P,P\rangle$ is positive. More precisely if we put
\[m=LCM(\{|G(F_{v})|:v\in B\})\]
then $\langle P,P\rangle \geq 1/m$ by \cite[Lem. 8.3]{Shioda_Mordell_Weil} and \cite[Thm. 8.4]{Shioda_Mordell_Weil}. The quantity $1/\langle P,P\rangle$ is bounded from above by $LCM(\{|G(F_{v})|:v\in B\})$ and
\[LCM(\{|G(F_{v})|:v\in B\})\leq 12 \prod_{v\in B_{mult,\geq 2}}m_{v},\]
where $B_{mult,\geq 2}$ denotes the set of places $v$ of multiplicative reduction and such that $m_{v}\geq 2$. We take the smallest possible $a\in\mathbb{R}$ such that for all integers $n\geq 2$ we have $n\leq a^{n}$. It implies that $a=\sup_{n\geq 2}n^{1/n}=3^{1/3}$. It follows from \eqref{eq:euler_formula} that
\[\prod_{v\in B_{mult,\geq 2}}m_{v}\leq a^{\sum_{v\in B_{mult,\geq 2}} m_{v}}\leq 3^{4\chi(S)}.\]
To finish the proof we apply \eqref{eq:height_def_comparision}.
\end{proof}

We define the \emph{conductor} of $E$ to be a divisor $N_{E}=\sum_{v\in C}u_{v}\:(v)$ where
\begin{equation*}
u_{v}=\left\{\begin{array}{ll}
              0 & \textrm{if the fibre at }v\textrm{ is smooth},\\
              1& \textrm{if the fibre at }v\textrm{ is multiplicative},\\
              2+\delta_{v}& \textrm{if the fibre at }v\textrm{ is additive},
             \end{array}\right.
\end{equation*}
and the nonnegative numbers $\delta_{v}$ are zero for $\charac K(C)\neq 2,3$. 
Let $j(E)$ denote the $j$--invariant of $E/K(C)$ treated as a function. When $j(E)$ is non--constant then let $p^{r}$ be its inseparable degree. If $\charac K(C)=0$, then we put $1$.
\begin{theorem}[{\cite[Thm. 0.1]{Pesenti_Szpiro}}]\label{thm:Pesenti_Szpiro}
Assume $E$ is an elliptic curve over $K(C)$. Let $p$ denote the characteristic of $K(C)$. When the map $j(E)$ is constant or $p=0$, then
\[\deg\Delta_{E}\leq 6 (2g(C)-2+\deg N_{E}).\]
When $j(E)$ is non--constant, $p>0$ and $p^{r}$ is its inseparable degree, then
\[\deg\Delta_{E}\leq 6p^{r} (2g(C)-2+\deg N_{E}).\]
\end{theorem}
We denote by $\sigma_{E}$ the so-called \emph{Szpiro ratio} which is defined as
\[\sigma_{E}=\frac{\deg \Delta_{E}}{\deg N_{E}}.\]

We denote by $LCM(1,2,\ldots, n)$ the least common multiple of all integers in the interval $[1,n]$.
\begin{theorem}[{\cite[Thm. 4.1]{SilvHindry}}]\label{thm:Silv_Hindry_height_raw_bound}
Let $E$ be an elliptic curve over $K(C)$ and let $P$ be a point of infinite order. Let $M\geq 1$, $N\geq 2$ be any integers. Then
\begin{equation*}
\heightE{P}\geq \frac{6\left(\left(1+\frac{1}{M}\right)\frac{1}{\sigma_{E}}-\frac{1}{M}-\frac{1}{N}\right)\cdot h_{K(C)}(E)}{(M+1)(M+2)\LCM(1,2,\ldots,N-1)^2}
\end{equation*}
\end{theorem}
The following fact is due to Rosser and Schoenfeld \cite{Rosser_Schoenfeld}. For the proof see \cite[Lem. 4.3]{SilvHindry}.
\begin{lemma}\label{lem:Chebyshev_function_bound}
 For all integers $n\geq 1$
 \[\log(\LCM(1,\ldots,n))<1.04 n.\]
\end{lemma}

We reproduce the main result of \cite{SilvHindry} with slightly corrected numerical constants.
\begin{theorem}[\protect{\cite[Thm. 6.1]{SilvHindry}}]\label{thm:SilvHindryBound}
Let $K(C)$ be a field of characteristic $0$. Let $P$ be a non--torsion point in $E(K(C))$. For $h_{K(C)}(E)\geq 2(g(C)-1)$ we have
\[\heightE{P}\geq 10^{-15.5}h_{K(C)}(E).\]
For $h_{K(C)}(E)<2(g(C)-1)$ we have
\[\heightE{P} \geq 10^{-9-23g(C)}h_{K(C)}(E).\]
\end{theorem}
\begin{proof}
From the first assumption and Theorem \ref{thm:Pesenti_Szpiro} it follows that $\sigma_{E}\leq 12$. To prove the first inequality we apply Theorem \ref{thm:Silv_Hindry_height_raw_bound} with $M=213$ and $N=13$. 

To prove the second statement we assume that $h_{K(C)}(E)<2(g(C)-1)$. Value $h_{K(C)}(E)$ is positive, so $g(C)\geq 2$. By assumption our curve has at least one place of bad reduction, hence $\deg N_{E}\geq 1$. The definition of $\sigma_{E}$ implies that
\[\sigma_{E}\leq 12 h_{K(C)}(E)<24g(C).\]
Let $M=601g(C)$ and $N=25g(C)$. We combine Theorem \ref{thm:Silv_Hindry_height_raw_bound} with Lemma \ref{lem:Chebyshev_function_bound}. It follows that
\[\frac{\heightE{P}}{h_{K(C)}(E)}\geq \frac{0.0016676 e^{-52 g(C)}}{g(C)^2 (300 g(C)+1) (600 g(C)+1)} \geq 10^{-9 - 23 g(C)}.\]
\end{proof}

We can now proceed in a similar way to obtain the analogue of Lang's conjecture for function fields $K(C)$ of positive characteristic. The bound is worse than in characteristic $0$ case, because we have to take into account the inseparable degree of the $j$--map.
\begin{lemma}\label{lem:height_est_char_p_1}
Let $P$ be a point of infinite order on $E$ over $K(C)$ of positive characteristic $p$ and assume that the $j$-map of $E$ has inseparable degree $p^r$. For $h_{K(C)}(E)\geq 2\cdot p^r (g(C)-1)$ we have
 \[\heightE{P}\geq 10^{-18 p^{r}}h_{K(C)}(E).\]
For $h_{K(C)}(E)< 2\cdot p^r (g(C)-1)$ it follows that
\[\heightE{P}\geq 10^{-36 g(C) p^{r}}h_{K(C)}(E).\]
\end{lemma}
\begin{proof}
Under the assumption $h_{K(C)}(E)\geq 2\cdot p^r (g(C)-1)$ Theorem \ref{thm:Pesenti_Szpiro} implies that
\[\frac{1}{\sigma_{E}}\geq \frac{1}{12p^r}.\]
Put $x=p^{r}$. We choose $M\geq 1$ and $N\geq 2$ such that 
\[\left(\left(1+\frac{1}{M}\right)\frac{1}{\sigma_{E}}-\frac{1}{M}-\frac{1}{N}\right)>0.\]
We take $M=200x^2$ and $N=12x+1$. Lemma \ref{lem:Chebyshev_function_bound} combined with Theorem \ref{thm:Silv_Hindry_height_raw_bound} implies that
\[\heightE{P}\geq \phi(x)h_{K(C)}(E)\]
where $\phi(x)=\frac{e^{-24.96 x} \left(56 x^2+1\right)}{800 x^3 (12 x+1) \left(100 x^2+1\right) \left(200 x^2+1\right)}$. For $x\geq 1$ we have the lower bound $\phi(x)\geq 10^{-18x}=10^{-18p^r}$.

We assume that $h_{K(C)}(E)< 2\cdot p^r (g(C)-1)$. Definition of $\sigma_{E}$ implies that $\sigma_{E}<24p^r (g(C)-1)<12x$ with $x=2g(C)p^r$. For $M$ and $N$ as before we obtain
\[\heightE{P}\geq \phi(x)h_{K(C)}(E)\]
with $\phi(x)\geq 10^{-36g(C)p^r}$.
\end{proof}
\begin{remark}
In positive characteristic and for constant $j$--map the bound on $\heightE{P}$ can be as good as in Theorem \ref{thm:SilvHindryBound}. For $K(C)$ with $\charac K(C)=0$ we can even prove that $\heightE{P}\geq \frac{1}{144}h_{K(C)}(E)$, cf. \cite[Thm. 6.1]{SilvHindry}. However, to simplify the statements, we don't make a distinction because the general weaker bounds apply as well.
\end{remark}

\section{Characteristic $0$ argument}\label{sec:char_0_section}
Let $\{D_{nP}\}_{n\in\mathbb{N}}$ be an elliptic divisibility sequence attached to a point $P$ in $E(K(C))$ of infinite order. Let $v$ denote a place in $K(C)$. Let $m(v)$ be a positive integer defined as follows
\[m(v):=\min\{n\geq 1: \ord_{v}(D_{nP})\geq 1 \}.\]
For a divisor $D_{nP}$ we define a new divisor $D_{nP}^{\textrm{new}}$ by the recipe
\[\ord_{v}D_{nP}^{\textrm{new}} =\left\{\begin{array}{cc}
   \ord_{v}D_{nP} &, m(v)=n\\
   0 &, \textrm{otherwise.}
  \end{array}\right.
\]
From this definition it follows by \cite[Lem. 5.6]{Silverman_divisibility} that
\begin{align*}
  D_{nP}&=\sum_{v\in \Supp D_{nP}} (\ord_{v}D_{nP})\: (v)\\
  &=\sum_{v\in \Supp D_{nP}}(\ord_{v}D_{m(v)P})\: (v) \quad(\textrm{from characteristic }0\textrm{ assumption})\\
  &=\sum_{\substack{v\in \Supp D_{nP}\\ m(v)<n}}(\ord_{v}D_{m(v)P})\: (v) + \sum_{v \in \Supp D_{nP}^{\textrm{new}}}(\ord_{v} D_{nP}^{\textrm{new}})\: (v)\\
  &\leq \sum_{\substack{m\mid n\\ m<n}} D_{mP} + D_{nP}^{\textrm{new}}
\end{align*}
It follows that for a divisor $D_{nP}$ which has no primitive valuations, i.e. such that $\Supp D_{nP}\subset \bigcup_{m < n} \Supp D_{mP}$ the following inequality 
\[D_{nP}\leq \sum_{\substack{m\mid n\\ m<n}} D_{mP}\]
holds. We apply the formula of Shioda for the height pairing to make the terms $O(1)$ from the proof of \cite[Thm. 5.5]{Silverman_divisibility} explicit. We rely fundamentally on the following estimate
\begin{equation}\label{eq:fund_estim}
\deg D_{nP}\leq \sum_{\substack{ %
   m\mid n \\
   m<n %
   }} \deg D_{mP}\qquad (\Longleftrightarrow)\qquad \dotprod{nP}{O}\leq \summy \dotprod{mP}{O}
\end{equation}
We define two quantities that will be used frequently
\begin{align*}
C_{1}(n,P) & = \frac{1}{2}\sum_{v\in B}c_{v}(nP,nP),\\
C_{2}(n,P) & = \frac{1}{2}\summy\sum_{v\in B}c_{v}(mP,mP).
\end{align*}
Assume $n>1$ and $D_{nP}$  is not primitive. We apply formulas \eqref{eq:Shioda_height} and \eqref{eq:fund_estim} to obtain the following chain of inequalities and equalities
\begin{align*}
&n^2\heightE{P} = \heightE{nP}=\frac{1}{2}\pairing{nP}{nP}\\
& = \dotprod{nP}{O}+\chi(S)-\frac{1}{2}\sum_{v\in B}c_{v}(nP,nP)\\
&\leq \summy\dotprod{mP}{O}+\chi(S)-\underbrace{\left(\frac{1}{2}\sum_{v\in B}c_{v}(nP,nP)\right)}_{C_{1}(n,P)}\\
&= \summy\left(\frac{1}{2}\pairing{mP}{mP}-\chi(S)+\frac{1}{2}\sum_{v\in B}c_{v}(mP,mP) \right)+ \chi(S)- C_{1}(n,P)\\
&=\frac{1}{2}\pairing{P}{P}\summy m^2 -\chi(S)\summy 1+\underbrace{\frac{1}{2}\summy\sum_{v\in B}c_{v}(mP,mP)}_{C_{2}(n,P)}+\chi(S)-C_{1}(n,P)\\
&=\frac{1}{2}\pairing{P}{P}(\sigma_{2}(n)-n^2)-\chi(S)(d(n)-2)+C_{2}(n,P)-C_{1}(n,P)\\
&=\heightE{P}(\sigma_{2}(n)-n^2)-\chi(S)(d(n)-2)+C_{2}(n,P)-C_{1}(n,P)
\end{align*}

\noindent
This can be rewritten in the following form
\begin{equation}\label{equation:fundamental_inequality}
\chi(S)(d(n)-2)+C_{1}(n,P)+n^2\heightE{P}\leq \heightE{P}(\sigma_{2}(n)-n^2)+C_{2}(n,P).
\end{equation}

\begin{lemma}\label{lemma:inequality1}
Let $P$ be a point of infinite order in $E(K(C))$ and let $n>1$ and assume $D_{nP}$ is not primitive. Then
\begin{equation}\label{equation:fund_inequality_simplified}
n^2\heightE{P}\leq \heightE{P}(\sigma_{2}(n)-n^2)+C_{2}(n,P)
\end{equation}
\end{lemma}
\begin{proof}
Since $n>1$ it is always true that $d(n)\geq 2$, the factor $\chi(S)$ is always positive and the terms in $C_{1}(n,P)$ are also non-negative by their definition. It implies that we can drop first two terms of the inequality \eqref{equation:fundamental_inequality}.
\end{proof}
Let $E(K(C))^{0}$ denote the subgroup of $E(K(C))$ such that for each $P\in E(K(C))^{0}$ the curve $\overline{P}$ intersects the same component as the curve $\overline{O}$ in every fibre of $\pi:S\rightarrow C$. For such points we always have $c_{v}(P,P)=0$.
\begin{corollary}
With the notation from the previous lemma if $P$ lies in $E(K(C))^{0}$, then every divisor $D_{nP}$ is primitive.
\end{corollary}
\begin{proof}
We use the inequality \eqref{equation:fund_inequality_simplified} and apply the assumption $C_{2}(n,P)=0$. It follows by \eqref{eq:sigma_2_estimate} that
\[n^2\heightE{P}\leq \heightE{P}(\zeta(2)-1)n^2.\]
We can divide by $\heightE{P}$ because $P$ is a point of infinite order, hence
\[2n^2\leq \zeta(2)n^2\]
and $n=0$.
\end{proof}

\begin{lemma}\label{lemma:inequality2_alt}
Let $K(C)$ be a field of characteristic $p\neq 2,3$. For a point $P\in E(K(C))$ and any $k\in\mathbb{Z}$ we have
\[\sum_{v\in B} c_{v}(kP,kP)\leq 3\chi(S).\]
\end{lemma}
\begin{proof}
We denote by $B_{mult}$ the set of points $v$ in $C(K)$ such that $F_{v}$ has multiplicative reduction. We denote by $B_{add,1}$ the set of points with additive reduction of type $I_{n}^{*}$ and by $B_{add,III}$, $B_{add,III^{*}}$, $B_{add,IV}$ and $B_{add,IV^{*}}$ the sets of points with respectively reduction of type $III$, $III^{*}$, $IV$ and $IV^{*}$. Let $B_{add,2}$ denote the set of all places of bad additive reduction not contained in $B_{add,1}$. Let $v\in B_{mult}$, then it follows from Figure \ref{fig:corrections} that
\[c_{v}(kP,kP)\leq \frac{i(m_{v}-i)}{m_{v}}\]
for certain $i$. The function on the right-hand side is quadratic with respect to $i$ and reaches the maximum at $m_{v}/2$, hence $c_{v}(kP,kP)\leq \frac{m_{v}}{4}$. That inequality and other values in Figure \ref{fig:corrections} allow us to give the upper bounds
\begin{align*}
\sum_{v\in B_{mult}} c_{v}(kP,kP)&\leq \frac{1}{4}\sum_{v\in B_{mult}} m_{v}\\
\sum_{v\in B_{add,III}}c_{v}(kP,kP)&\leq \frac{1}{2}| B_{add,III} |\\
\sum_{v\in B_{add,III^{*}}}c_{v}(kP,kP)&\leq \frac{3}{2}| B_{add,III^{*}} |\\
\sum_{v\in B_{add,IV}}c_{v}(kP,kP)&\leq \frac{2}{3}| B_{add,IV} |\\
\sum_{v\in B_{add,IV^{*}}}c_{v}(kP,kP)&\leq \frac{4}{3}| B_{add,IV^{*}} |
\end{align*}
For points $v$ of type $B_{add,1}$ we have $c_{v}(kP,kP)\leq \frac{m_{v}-1}{4}=\frac{m_{v}+1}{4}-\frac{1}{2}$. This leads to
\[2\cdot |B_{add,1}|+4\sum_{v\in B_{add,1}}c_{v}(kP,kP)\leq \sum_{v\in B_{add,1}} (m_{v}+1).\]
It follows from \eqref{eq:euler_formula} that
\[12\chi(S)=\sum_{v\in B}e(F_{v})=\sum_{v\in B_{mult}} m_{v}+\sum_{v\in B_{add,1}} (m_{v}+1)+\sum_{v\in B_{add,2}}(m_{v}+1).\]
But we also have
\[\sum_{v\in B_{add,2}}(m_{v}+1)=3\cdot| B_{add,III} |+9\cdot | B_{add,III^{*}} |+4\cdot| B_{add,IV} |+8\cdot| B_{add,IV^{*}} |\]
by \cite[Chap. IV, Table 4.1]{Silverman_book}. It follows that
\[12\chi(S)\geq 4\sum_{v\in B_{mult}} c_{v}(kP,kP)+4\sum_{v\in B_{add,1}} c_{v}(kP,kP)+6\sum_{v\in B_{add,2}} c_{v}(kP,kP)\]
which is even stronger than what we wanted to prove.
\end{proof}
\begin{remark}
The statement of Lemma \ref{lemma:inequality2_alt} is equivalent to \cite[Lem. 3]{Elkies_low_height}. The upper bound in loc. cit. follows from \eqref{eq:Shioda_height}.
\end{remark}

\begin{lemma}\label{lemma:inequality2}
Let $K(C)$ be a field of characteristic $0$. Let $P$ be a point in $E(K(C))$. Then 
\begin{equation*}
C_{2}(n,P)\leq \frac{3}{2} \chi(S)(d(n)-1).
\end{equation*}
\end{lemma}
\begin{proof}
This follows simply from the definition of $C_{2}(n,P)$ and Lemma \ref{lemma:inequality2_alt}.
\end{proof}

\begin{corollary}\label{cor:dep_on_chi}
Let $P$ be a point of infinite order in $E(K(C))$. Suppose that $D_{nP}$ is not primitive, then
\[n^2\leq \frac{36\cdot  \chi(S) \cdot 3^{4\chi(S)}}{(2-\zeta(2))}d(n)\]
\end{corollary}
\begin{proof}
Combine Lemmas \ref{lemma:inequality3}, \ref{lemma:inequality1} and \ref{lemma:inequality2}. 
\end{proof}
\begin{corollary}\label{cor:bound_silv}
Let $K(C)$ be a field of characteristic $0$. Let $P$ be a point of infinite order in $E(K(C))$. If $D_{nP}$ is not primitive, then
\begin{equation*}
n^2\leq \frac{1.5\cdot 10^{9}}{(2-\zeta(2))}d(n)\cdot \left\{\begin{array}{ll} 10^{6.5}  & , \chi(S)\geq 2(g(C)-1)\\
 10^{23g(C)} & , \chi(S)< 2(g(C)-1)
\end{array}\right.
\end{equation*}
\end{corollary}
\begin{proof}
To bound the quantity $1/\heightE{P}$ we apply Theorem \ref{thm:SilvHindryBound}. Suppose that  $\chi(S)\geq 2(g(C)-1)$, then
\[1/\heightE{P}\leq 10^{15.5}\cdot 1/\chi(S)\]
Combining this with the argument in Lemma \ref{lemma:inequality2} we obtain
\[1/\heightE{P}\cdot C_{2}(n,P)\leq 10^{15.5}\cdot 1/\chi(S)\cdot 1.5\cdot\chi(S)\cdot d(n) = 1.5\cdot 10^{15.5}  d(n).\]
It follows that
\begin{equation}\label{eq:estim2}
n^2\leq (1.5\cdot 10^{15.5})/(2-\zeta(2))\cdot d(n).
\end{equation}
On the contrary, when $\chi(S)<2(g(C)-1)$ we get
\[1/\heightE{P}\cdot C_{2}(n,P)\leq 10^{9+23g(C)}\cdot 1/\chi(S)\cdot 1.5\cdot\chi(S)\cdot d(n)=1.5\cdot 10^{9+23g(C)}d(n).\]
Similarly, we get
\begin{equation}\label{eq:estim3}
n^2\leq (1.5\cdot 10^{9+23g(C)})/(2-\zeta(2))\cdot d(n).
\end{equation}
The corollary follows from those two estimates.
\end{proof}

\begin{proof}[Proof of Theorem \ref{thm:dep_on_g}]
We have the trivial estimate $d(n)\leq n$. Corollary \ref{cor:bound_silv} implies that
\[n^2\leq C n\]
for a constant $C$ that depends only on $g(C)$. So $n\leq C$ and the theorem follows.
\end{proof}

\begin{proof}[Proof of Theorem \ref{thm:dep_on_chi}]
There exists a constant $C$ that depends only on $\chi(S)$ as in Corollary \ref{cor:dep_on_chi} such that $n^2\leq C n$.
\end{proof}

\begin{remark}
If we assume that $n\geq N_{0}$ where $N_{0}$ is sufficiently large, we obtain due to \eqref{eq:d_bound} a much better bound for $d(n)$. This will lead in practice to a much smaller bound for the number of non--primitive divisors.
\end{remark}
\section{Characteristic p argument}\label{sec:positive_char}
Let $v$ be a discrete valuation on $K(C)$. It determines the completion $K(C)_{v}$ of the field $K(C)$ with respect to $v$ with ring of integers $R_{v}$ and maximal ideal $\mathcal{M}_{v}$. We consider below only fields $K(C)$ of characteristic at least $5$. For an elliptic curve $E$ over $K(C)$ we consider its minimal Weierstrass model $E^{(v)}$ at $v$, cf. \cite[Chap. VII, \S 1]{Silverman_arithmetic}. Such a model is unique up to an admissible change of coordinates, cf. \cite[Chap. VII, Prop. 1.3]{Silverman_arithmetic}. We denote by $\widehat{E}^{(v)}$the formal group attached to the minimal Weierstrass equation $E^{(v)}$ in the sense of \cite[Chap. IV]{Silverman_arithmetic}. Multiplication by $p$ map gives rise to a homomorphism of formal groups $\widehat{[p]}_{v}:\widehat{E}^{(v)}\rightarrow\widehat{E}^{(v)}$. Its height $h$ equals $1$ or $2$, cf.\cite[Chap. IV, Thm. 7.4]{Silverman_arithmetic}. If the height equals $h$, then $\widehat{[p]}_{v}(T)=g(T^{p^h})$ where $g(T)\in R_{v}[[T]]$ and $g'(0)\neq 0$. The coefficient of $T^{p}$ in $\widehat{[p]}_{v}(T)$ is denoted by $H(E,v)$ and is the Hasse invariant in the sense of \cite[12.4]{Katz_Mazur}. The valuation $h_{E,v}:=\ord_{v}(H(E,v))$ does not depend of the minimal model at $v$ by \cite[Ka-29]{Katz_p_adic}. We say that the curve $E$ is \emph{ordinary} when for all discrete valuations $v$ of $K(C)$ the homomorphism $\widehat{[p]}_{v}$ has height $1$. 

\begin{lemma}\label{lem:Valutions_of_h_E_v}
Let $E$ over $K(C)$ of characteristic $p>3$ be an ordinary elliptic curve and let $\chi(S)$ denote the Euler characteristic of the attached elliptic surface $\pi:S\rightarrow C$. Then
\[(p-1)\chi(S)=\sum_{v\in C}h_{E,v}.\]
\end{lemma}
\begin{proof}
For any place $v$ in $K(C)$ we fix a minimal model $E^{(v)}$ of $E$ at $v$ with Hasse invariant $H(E,v)$. Let $\Delta\in K(C)$ be the discriminant and let $H(E)\in K(C)$ denote the Hasse invariant of one arbitrarily chosen model $E^{(v_{0})}$ at $v_{0}$. We denote by $\Delta_{v}$ the minimal discriminant of $E$ at $v$. For each $v$ there exists an integer $n_{v}$ such that
\begin{equation}\label{eq:local_delta_trans}
\ord_{v}(\Delta)=\ord_{v}(\Delta_{v})+12n_{v}.
\end{equation}
From \cite[Ka-29]{Katz_p_adic} it follows that
\begin{equation}\label{eq:local_Hasse_trans}
\ord_{v}(H(E))=\ord_{v}(H(E,v))+(p-1)n_{v}. 
\end{equation}
Elements $\Delta$ and $H(E)$ correspond to functions $\Delta, H(E):C\rightarrow\mathbb{P}^{1}$ and hence $\sum_{v\in C}\ord_{v}(H(E))=\sum_{v\in C}\ord_{v}(\Delta)=0$. Summation over all $v$ combined with \eqref{eq:local_delta_trans} and \eqref{eq:local_Hasse_trans} implies that
\[\frac{(p-1)\sum_{v\in C}\ord_{v}\Delta_{v}}{12}=\sum_{v\in C}h_{E,v}.\]
To finish the proof we apply $12\chi(S)=\sum_{v\in C}e(F_{v})=\sum_{v\in C}\ord_{v}\Delta_{v}$.
\end{proof}

We generalise \cite[Lemma 5.6]{Silverman_divisibility} to the case of positive characteristic. We note that a similar lemma can be obtained in the number field case, cf. \cite{Stange}.
\begin{lemma}\label{lemma:order_char_p}
Let $E$ be an ordinary elliptic curve over $K(C)$, field of characteristic $p$. Let $\{D_{nP}\}_{n\in\mathbb{N}}$ be an elliptic divisibility sequence attached to a point $P$ in $E(K(C))$ of infinite order. Let $v$ denote a place in $K(C)$. Let $m(v)$ be a positive integer defined as follows
\[m(v):=\min\{n\geq 1: \ord_{v}(D_{nP})\geq 1 \}.\]
If $h_{E,v}\leq p-1$, then for all $n\geq 1$ the following equality
\begin{equation*}
 \ord_{v}D_{nP}=\left\{\begin{array}{ll}
                 p^e\ord_{v}D_{m(v)P}+\frac{p^e -1}{p-1}h_{E,v} & ,m(v)\mid n\\
                 0 & ,m(v)\nmid n
                \end{array}\right.
\end{equation*}
holds for $e=v_{p}(\frac{n}{m(v)})$.

Let $k\geq\lceil\log_{p} (\frac{p+(p-1)^2\chi(S)}{2p-1})\rceil$ be an integer. For $h_{E,v}\geq p$ and for all $n\geq 1$ the following equality
\begin{equation*}
 \ord_{v}D_{nP}=\left\{\begin{array}{ll}
                 p^{e}\ord_{v}D_{m(v)P}+\delta(e) & ,m(v)\mid n,\ e\leq k\\
                 p^{e}\ord_{v}D_{m(v)P}+\frac{p^{e-k} -1}{p-1}h_{E,v}+p^{e-k}\delta(k) & ,m(v)\mid n,\ e> k\\
                 0 & ,m(v)\nmid n
                \end{array}\right.
\end{equation*}
holds for $e=v_{p}(\frac{n}{m(v)})$. Function $\delta(e)$ depends on $P$ and $v$ and satisfies the estimates for $e\geq 1$
\[p\cdot \frac{p^{e}-1}{p-1}\leq \delta(e)\leq p^{2e}m(v)^2 \heightE{P}+\frac{1}{2}\chi(S)-p^{e}.\]
\end{lemma}
\begin{proof}
Let $E(K(C))_{v,r}$ denote the set \[\{P\in E(K(C)): \ord_{v}\sigma_{P}^{*}\overline{O} \geq r\}\cup \{\mathcal{O}\}.\] It follows from its definition that $E(K(C))_{v,r}$ is a subgroup of $E(K(C))$. Number $\ord_{v}D_{nP}$ equals $\max\{r\geq 0: nP\in E(K(C))_{v,r}\}$. We consider the completion $K(C)_{v}$ of field $K(C)$ with respect to $v$, with integer ring $\mathcal{R}_{v}$ and maximal ideal $\mathcal{M}_{v}$. 
Suppose that $d_{0}:=\ord_{v}D_{m(v)P}$ and $d:=\ord_{v}D_{nP}\geq 1$. The subgroups $\{E(K(C))_{v,r}\}_{r\geq 1}$ form a nested sequence so 
\[\GCD(m(v),n)P\in E(K(C))_{v,\min\{d_{0},d\}}.\]
Minimality of $m(v)$ implies that $m(v)\leq \GCD(m(v),n)$, hence $m(v)\mid n$. 

By \cite[Chap.VII, Prop. 2.2]{Silverman_arithmetic} there exists an isomorphism
\[i_{v}:E_{1}(K(C)_{v})\rightarrow \widehat{E}(\mathcal{M}_{v})\]
given by $(x,y)\rightarrow -x/y$ and where $E_{1}(K(C)_{v})$ is the kernel of reduction at $v$ defined in \cite[Chap.VII]{Silverman_arithmetic}. We note that the group $E(K(C))_{v,1}$ is a subgroup of $E_{1,v}(K(C)_{v})$. For an integer $n$ coprime to $p$ and $P\in E(K(C))_{v,1}$ we have
\[\ord_{v}(i_{v}(nP))=\ord_{v}(i_{v}(P)).\]
Assume that $\ord_{v}(h_{E,v})\leq p-1$. It follows that $\ord_{v}(i_{v}(pP))=h_{E,v}+p\ord(i_{v}(P))$. By iteration we obtain
\[\ord_{v}(i_{v}(nP))=p^{e}\ord_{v}(i_{v}(P))+h_{E,v}(1+\ldots +p^{e-1})\]
where $e=v_{p}(n)$. 

For $\ord_{v}(h_{E,v})\geq p$ and for any $P \in E(K(C))_{v,1}$ we have $\ord_{v}(i_{v}(pP))\geq p+p\ord(i_{v}(P))$. After $e$ iterations this implies that
\[\ord_{v}(i_{v}(p^{e}P))\geq p\cdot \frac{p^{e}-1}{p-1}+p^{e}\ord(i_{v}(P)).\]
The formal group homomorphism $\widehat{[p]}_{v}$ satisfies $\ord_{v}(\widehat{[p]}_{v}(T))=h_{E,v}+p\ord_{v}(T)$ for $T$ such that $\ord_{v}(T)>h_{E,v}$. Lemma \ref{lem:Valutions_of_h_E_v} implies that $h_{E,v}\leq (p-1)\chi(S)$. If $e$ is greater than $k$, then we have
\[p^{e}+\frac{p^{e}-1}{p-1}\cdot p > (p-1)\chi(S).\]
Thus $\ord_{v} i_{v}(p^{e}P)=p^{e}\ord_{v}(i_{v}(P))+h_{E,v}(1+\ldots +p^{e-k-1})+\delta(k)$ where $\delta(k)=\ord_{v}i_{v}(p^{k}P)-p^{k}\ord_{v}i_{v}(P)$. 

For any $e\leq k$ we define $\delta(e)=\ord_{v} i_{v}(p^{e}P)-p^{e}\ord_{v}i_{v}(P)$. It is clear that $\delta(e)\geq p\cdot \frac{p^{e}-1}{p-1}$. For the upper bound we observe that
\[p^{2e} m(v)^2\heightE{P}+\frac{1}{2}\chi(S)\geq \ord_{v}D_{p^{e}m(v)P} =\ord_{v}D_{nP}\]
by property \eqref{eq:Shioda_height} and Lemma \ref{lemma:inequality2_alt}. Since $\ord_{v}D_{m(v)P}\geq 1$, the upper bound follows by replacing $P$ by $m(v)P$ in the definition of $\delta(e)$.
\end{proof}

\begin{definition}\label{def:tame_wild_def}
Let $E$ be an ordinary elliptic curve over a function field $K(C)$ of prime characteristic $p$. We say that $E$ is \emph{tame}, when for all places $v$ we have $h_{E,v}\leq p-1$. Otherwise we say that $E$ is \emph{wild}.
\end{definition}

If $\charac K(C)=p> 0$ we apply Lemma \ref{lemma:order_char_p} instead of \cite[Lemma 5.6]{Silverman_divisibility}. Under assumption that $D_{nP}$ has no primitive valuations it follows that
\begin{align*}
  D_{nP}&=\sum_{v\in \Supp D_{nP}}(\ord_{v}D_{nP})\: (v)\\
  &=\sum_{\substack{v\in \Supp D_{nP}\\ m(v)<n}}(\ord_{v}D_{nP})\: (v) + \sum_{v \in \Supp D_{nP}^{\textrm{new}}}(\ord_{v} D_{nP}^{\textrm{new}})\: (v)\\
  &=\sum_{\substack{v\in \Supp D_{nP}\\ m(v)<n}}(\ord_{v}D_{nP})\: (v) \quad(\textrm{no primitive valuations})\\
  &=\sum_{\substack{v\in \Supp D_{nP}\\ m(v)<n}}(p^{v_{p}(\frac{n}{m(v)})}\ord_{v}D_{m(v)P})\: (v)+\underbrace{\sum_{\substack{v\in \Supp D_{nP}\\ m(v)<n}}f(E,P,n,v)\: (v)}_{W(E,P,n)}\\
  &\leq \sum_{\substack{m\mid n\\ m<n}}\sum_{v\in C}(p^{v_{p}(\frac{n}{m})}\ord_{v}D_{mP})\: (v)+ W(E,P,n)\\
  &= \sum_{\substack{m\mid n\\ m<n}}p^{v_{p}(\frac{n}{m})}D_{mP}+ W(E,P,n).\\
\end{align*}
Function $f(E,P,n,v)$ is defined as the difference 
\[f(E,P,n,v)=\ord_{v}D_{nP}-p^{v_{p}(\frac{n}{m(v)})}\ord_{v}D_{m(v)P}.\]
We can summarize the computations above in the following corollary.
\begin{corollary}
Let $p>3$ be a prime number. Let $E$ be an ordinary elliptic curve over $K(C)$ and let $P$ be a point of infinite order on $E$. Assume $n$ is such that $D_{nP}$ is a divisor without primitive valuations. When $p\nmid n$, then
\begin{equation*}
 D_{nP}\leq \sum_{\substack{m\mid n\\ m<n}} D_{mP}.
\end{equation*}
When $\charac K(C)=p$, $p\mid n$, $n=n_{0}p^{e}$ and $p\nmid n_{0}$, then
\begin{equation}\label{eq:fund_ineq_char_p}
 D_{nP}\leq \sum_{\substack{m\mid n\\ m<n}}p^{v_{p}(\frac{n}{m})}D_{mP}+W(E,P,n)
\end{equation}
\end{corollary}
We apply the degree function to \eqref{eq:fund_ineq_char_p}. If $n$ is such that $D_{nP}$ has no primitive divisors and $p\mid n$ ($p>3$), then
\begin{align*}
  \overline{nP}.\overline{O}\leq \sum_{\substack{m\mid n\\ m<n}}p^{v_{p}(\frac{n}{m})}\overline{mP}.\overline{O}+\deg W(E,P,n).
\end{align*}
Now we redo the computations from characteristic $0$
\begin{align*}
&n^2\heightE{P} = \heightE{nP}=\frac{1}{2}\pairing{nP}{nP}\\
& = \dotprod{nP}{O}+\chi(S)-\frac{1}{2}\sum_{v\in B}c_{v}(nP,nP)\\
&\leq \sum_{\substack{m\mid n\\ m<n}}p^{v_{p}(\frac{n}{m})}\overline{mP}.\overline{O}+\underbrace{\deg W(E,P,n)}_{C_{3}(n,p,P)}+\chi(S)-\underbrace{\left(\frac{1}{2}\sum_{v\in B}c_{v}(nP,nP)\right)}_{C_{1}(n,P)}\\
&= \sum_{\substack{m\mid n\\ m<n}}p^{v_{p}(\frac{n}{m})}\left(\frac{1}{2}\pairing{mP}{mP}-\chi(S)+\frac{1}{2}\sum_{v\in B}c_{v}(mP,mP) \right)\\
&+ C_{3}(n,p,P)+\chi(S)- C_{1}(n,P)\\
&=\frac{1}{2}\pairing{P}{P}\summy p^{v_{p}(\frac{n}{m})} m^2 -\chi(S)\summy p^{v_{p}(\frac{n}{m})}\\
&+\underbrace{\frac{1}{2}\summy p^{v_{p}(\frac{n}{m})}\sum_{v\in B}c_{v}(mP,mP)}_{C_{2}(n,p,P)}+C_{3}(n,p,P)+\chi(S)-C_{1}(n,P)\\
&=\heightE{P}(\sigma_{2}^{(p)}(n)-n^2)-\chi(S)(d^{(p)}(n)-2)+C_{2}(n,p,P)\\
&\quad+C_{3}(n,p,P)-C_{1}(n,P)
\end{align*}

\begin{lemma}\label{lemma:fundamental_inequality2}
Let $p>3$ be a prime and let $\charac K(C)=p$. Let $P$ be a point of infinite order in $E(K(C))$ and let $n>1$ and assume $D_{nP}$ is not primitive. When $p\nmid n$ then
\begin{equation*}
n^2\heightE{P}\leq \heightE{P}(\sigma_{2}(n)-n^2)+C_{2}(n,P)
\end{equation*}
When $p\mid n$ then
\begin{equation*}
n^2\heightE{P}\leq \heightE{P}(\sigma_{2}^{(p)}(n)-n^2)+C_{2}(n,p,P)+C_{3}(n,p,P)
\end{equation*}
\end{lemma}
\begin{proof}
For $n$ coprime with $p$ Lemma \ref{lemma:order_char_p} implies that our inequalities reduce to the situation known from characteristic $0$. Assume now that $p\mid n$. Since $n>1$ it is always true that $d^{(p)}(n)\geq 2$, the factor $\chi(S)$ is always positive and the terms in $C_{1}(n,p,P)$ are also non-negative by their definition and the lemma follows.
\end{proof}

We need to establish some crude estimates of $C_{2}(n,p,P)$ and $C_{3}(n,p,P)$. 
\begin{lemma}\label{lem:C2_C3_estimates}
Let $p>3$ be a prime and let $\charac K(C)=p$. Let $P$ be a point of infinite order in $E(K(C))$ and let $n>1$ and assume $D_{nP}$ is not primitive. We obtain the estimate
\begin{equation}\label{eq:inequality_char_p_2}
C_{2}(n,p,P)\leq \frac{3}{2}\chi(S)\cdot(d^{(p)}(n)-1). 
\end{equation}
\end{lemma}
\begin{proof}
We apply Lemma \ref{lemma:inequality2_alt} to prove the inequality \eqref{eq:inequality_char_p_2}. 
\end{proof}
To get a uniform result we have to estimate the sum $W(E,P,n)$ independently of $n$. To achieve this we prove a technical lemma.
\begin{lemma}\label{lem:supersing_char_at_fibers}
Let $E$ and $P$ be given. Let $v$ denote a place in $K(C)$ and assume $h_{E,v}>0$. 

Then one of the cases holds
\begin{itemize}
 \item $E$ at $v$ has good reduction and then $p\nmid m(v)$.
 \item $E$ at $v$ has additive reduction and then $m(v)\mid 12p$.
 \item $E$ at $v$ has multiplicative reduction and $h_{E,v}>0$ cannot both occur.
\end{itemize}
\end{lemma}
\begin{proof}
Assume first that $E$ has good reduction at $v$. The assumption $h_{E,v}>0$ implies that locally at $v$ the fibre $E_{v}$ satisfies $E_{v}[p]=0$ by \cite[Chap.V, Thm. 3.1]{Silverman_arithmetic}. If $p\mid m(v)$, then $(m(v)/p)P$ would already meet the zero section at $v$ contradicting the minimality of $m(v)$. 

If $v$ is of additive reduction, then from Kodaira classification of bad fibres, cf. \cite[Chap. IV, Table 4.1]{Silverman_book} it follows that there exists an integer $k\in\{1,2,3,4\}$ such that the point $kP$ hits the component of zero at $v$. Either $kP$ is zero locally at $v$ or $pkP$ is zero. It implies that $m(v)\mid 12p$. 

Let $t$ be a formal variable and consider the series with coefficients in $\mathbb{Z}[[t]]$ as in \cite{Tate_param}
\begin{align*}
b_{2}(t)&=5\sum_{n=1}^{\infty}\frac{n^3 t^n}{1-t^n} = 5t+45t^2+140t^3+\ldots\\
b_{3}(t)&=\sum_{n=1}^{\infty}\left(\frac{7n^5+5n^3}{12}\right)\frac{t^n}{1-t^n}=t+23t^2+154t^3+\ldots\\
\Delta(t)&=b_{3}+b_{2}^2+72b_{2}b_{3}-432b_{3}^2+64b_{2}^{3}=t\prod_{n=1}^{\infty}(1-t^n)^{24}\\
j(t)&=\frac{(1+48b_{2})^3}{\Delta}=\frac{1}{t}(1+744t+196884t^2+\ldots)
\end{align*}

Finally, let $E$ at $v$ have multiplicative reduction. The normalised $v$-adic norm of $j(E)$ is greater than $1$. There exists a parameter $q\in \mathcal{M}_{v}$ such that $j(E)=j(q)$ (\cite[\S 3, VII]{Roquette_P_Analytic}) and the curve
\[E_{q}: y^2+xy=x^3-b_{2}(q)x+b_{3}(q)\]
has $j$--invariant equal to $j(q)$, has discriminant $\Delta(q)$ and is an elliptic curve over $K(C)_{v}$. It follows that 
\begin{align*}
c_{4}(E_{q})&=1+240\sum_{n=1}^{\infty}q^n\sum_{m \mid n}m^3,\\
c_{6}(E_{q})&=-1+504\sum_{n=1}^{\infty}q^n\sum_{m \mid n}m^5.\\
\end{align*}
It implies that the Weierstrass model $E_{q}$ is minimal at $v$ and the curves $E$ and $E_{q}$ are isomorphic over some extension $L$ of $K(C)_{v}$. The isomorphism corresponds to a change of coordinates between a minimal Weierstrass model of $E$ (with coordinates $x'$ and $y'$) and $E_{q}$ with $x\mapsto u^2 x'+r$, $y\mapsto u^3 y'+u^2 s x'+t$ where $u,s,t$ belong to the ring of integers of $L$. We have also $\ord_{v}(u)=0$ so the equality $h_{E_{q},v}=h_{E,v}$ holds by \cite[Ka-29]{Katz_p_adic}. By \cite[Thm. 12.4.2]{Katz_Mazur} we have $h_{E_{q},v}=0$, which contradicts our assumption $h_{E,v}>0$.
\end{proof}

For the next two lemmas assume that $E$ is an ordinary elliptic curve over $K(C)$, field of characteristic $p>3$. Let $\{D_{nP}\}_{n\in\mathbb{N}}$ be an elliptic divisibility sequence attached to a point $P$ in $E(K(C))$ of infinite order and let $S\rightarrow C$ be an elliptic surface corresponding to $E$. We denote by $e$ the $p$--valuation $v_{p}(n)$ of $n$.
\begin{lemma}\label{lemma:tame_case}
Let $E$ be tame. Then
\[\deg W(E,P,n)\leq (p^{e}-1)\chi(S).\]
\end{lemma}
\begin{proof}
In the tame situation we have $f(E,P,n,v)\leq\frac{p^e -1}{p-1}h_{E,v}$. Combination of this equality with Lemma \ref{lem:Valutions_of_h_E_v} proves the statement.
\end{proof}

Let $\mathcal{R}=\mathcal{R}(P,n)=\{v:v\in\Supp D_{nP},\  m(v)<n\}$. Denote by $\Sigma_{g}$ and $\Sigma_{a}$ the set of places of respectively good and bad additive reduction of $E$. Let $\mathcal{R}_{g}=\mathcal{R}\cap\Sigma_{g}$ and $\mathcal{R}_{a}=\mathcal{R}\cap\Sigma_{a}$. Let $\mathcal{S}$ denote the set of places $v$ in $K(C)$ such that $h_{E,v}>0$. Let $\Sigma_{g}^{s}=\Sigma_{g}\cap \mathcal{S}$ and $\Sigma_{a}^{s}=\Sigma_{a}\cap \mathcal{S}$.
\begin{lemma}\label{lemma:wild_case}
Let $E$ be wild and
let $M$ denote $\max\{144p^2,\max_{v\in\mathcal{R}_{g}\cap\mathcal{S}} m(v)^{2}  \}$. The following estimates hold
for any $n$ and $P$ of infinite order
\begin{itemize}
	\item[(i)] For $v_{p}(n)\leq \lceil\log_{p} (\frac{p+(p-1)^2\chi(S)}{2p-1})\rceil$ we have
	\[\deg W(E,P,n)\leq (p^{e}-1)\chi(S) +\chi(S)p^{2e}\heightE{P}M+\frac{1}{2}\chi(S)^2 .\]
	\item[(ii)] For $v_{p}(n)>\lceil\log_{p} (\frac{p+(p-1)^2\chi(S)}{2p-1})\rceil$ we have
	\[\deg W(E,P,n)\leq \chi(S)\left((p^{e}-1)+p^{e-k}(1 +(p^{2k}M\heightE{P}+\frac{1}{2}\chi(S)))\right).\]
\end{itemize}

\end{lemma}
\begin{proof}
From Lemma \ref{lemma:order_char_p} we can split the expression $\deg W(E,P,n)$ into two parts and estimate them separately.

\begin{align*}
\deg W(E,P,n)&=\sum_{v\in\mathcal{R}\cap\mathcal{S}}f(E,P,n,v)\\
&=\sum_{h_{E,v}<p}f(E,P,n,v)+\sum_{h_{E,v}\geq p}f(E,P,n,v)\\
&\leq (p^{e}-1)\chi(S)+\sum_{h_{E,v}\geq p}f(E,P,n,v).
\end{align*}
The last inequality follows from $f(E,P,n,v)\leq\frac{p^e -1}{p-1}h_{E,v}$ for $h_{E,v}<p$ and Lemma \ref{lem:Valutions_of_h_E_v}. Put $k=\lceil\log_{p} (\frac{p+(p-1)^2\chi(S)}{2p-1})\rceil$ and assume that $e=v_{p}(n)\leq k$. It follows that
\[f(E,P,n,v)\leq p^{2e}m(v)^2 \heightE{P}+\frac{1}{2}\chi(S)-p^{e}\]
for $v$ such that $h_{E,v}\geq p$. By Lemma \ref{lem:Valutions_of_h_E_v} there is at most $\frac{p-1}{p}\chi(S)$ such different places $v$. By Lemma \ref{lem:supersing_char_at_fibers} they can be only of good or additive reduction. Hence
\begin{align*}
\sum_{h_{E,v}\geq p}f(E,P,n,v)&\leq \frac{p-1}{p}\chi(S)p^{2e}\heightE{P}(\max\{\max_{v\in\Sigma_{a}^{s}}m(v)^{2},\max_{v\in\mathcal{R}_{g}\cap\mathcal{S}} m(v)^{2}  \})\\
&+\frac{p-1}{p}\chi(S)(\frac{1}{2}\chi(S)-p^{e}).
\end{align*}
By Lemma \ref{lem:supersing_char_at_fibers} it follows that $\max_{v\in\Sigma_{a}^{s}}m(v)\leq 12p$, hence
\[\sum_{h_{E,v}\geq p}f(E,P,n,v)\leq \chi(S)p^{2e}\heightE{P}M+\frac{1}{2}\chi(S)^2. \]
Assume now that $e> k$. We have the inequality
\[f(E,P,n,v)\leq \frac{p^{e-k} -1}{p-1}h_{E,v}+p^{e-k}\delta(k)\]
where $\delta(k) \leq p^{2k}m(v)^2 \heightE{P}+\frac{1}{2}\chi(S)-p^{k}$.
It implies that
\begin{align*}
\sum_{h_{E,v}\geq p}f(E,P,n,v)&\leq (p^{e-k}-1)\chi(S) +\\
&\frac{p-1}{p}\chi(S)p^{e-k}(p^{2k}M\heightE{P}+\frac{1}{2}\chi(S)-p^{k})
\end{align*}
or in simplified form
\[\sum_{h_{E,v}\geq p}f(E,P,n,v)\leq p^{e-k}\chi(S) +p^{e-k}\chi(S)(p^{2k}M\heightE{P}+\frac{1}{2}\chi(S)).\]
\end{proof}
\begin{remark}
We observe that the bound $\lceil\log_{p} (\frac{p+(p-1)^2\chi(S)}{2p-1})\rceil$ approaches $1$ as $p\rightarrow\infty$ independently of $\chi(S)$.
\end{remark}

\begin{theorem}\label{thm:main_theorem_tame}
Let $E$ be an elliptic curve over $K(C)$ of positive characteristic $p>3$ with at least one bad fibre. Assume that $E$ is tame.  Let $\pi:S\rightarrow C$ be the attached elliptic fibration.  Let $P$ be a point of infinite order on $E$. Let $p^{r}$ be the inseparable degree of the $j$--map of $E$. There exists an explicit constant $N=N(g(C),p,r)$ which depends only on the genus of $C$, $p$ and $r$ such that for all $n\geq N$ the divisor $D_{nP}$ has a primitive valuation.
\end{theorem}
\begin{proof}
Let $n$ be an integer such that the divisor $D_{nP}$ has no primitive valuation. Let us first assume that $p\nmid n$. Lemma \ref{lemma:fundamental_inequality2} implies that
\[n^2\heightE{P}\leq \heightE{P}(\sigma_{2}(n)-n^2)+C_{2}(n,P).\]
We combine the estimate $\sigma_{2}(n)<\zeta(2)n^2$ with the estimate from Lemma \ref{lemma:inequality2}. The only difference with characteristic zero case is that we apply now the height estimate for $\heightE{P}$ from Lemma \ref{lem:height_est_char_p_1}. It follows that there exists an effective constant $N_{1}=N_{1}(g(C),p^{r})$ such that $n\leq N_{1}$.

Let us assume that $p\mid n$. After Lemma \ref{lemma:fundamental_inequality2} we have
\[n^2\heightE{P}\leq \heightE{P}(\sigma_{2}^{(p)}(n)-n^2)+C_{2}(n,p,P)+C_{3}(n,p,P).\]
By Proposition \ref{prop:arithm_estimates} it follows that
\[n^2\heightE{P}\leq \heightE{P}\cdot \left(\left(1+\frac{1}{p}\right)\zeta(2)-1\right)n^2+C_{2}(n,p,P)+C_{3}(n,p,P)\]
and in simplified form
\[\theta(p) n^2\heightE{P}\leq C_{2}(n,p,P)+C_{3}(n,p,P)\]
where by $\theta(p)$ we denote $2-\left(1+\frac{1}{p}\right)\zeta(2)$. We apply Lemma \ref{lem:C2_C3_estimates} and get the bound
\begin{equation*}
\theta(p) n^2\heightE{P}\leq\frac{3}{2}\chi(S)\cdot(d^{(p)}(n)-1)+\deg W(E,P,n).
\end{equation*}
Put $e=v_{p}(n)$. Lemma \ref{lemma:tame_case} implies that $\deg W(E,P,n)\leq (p^e-1)\chi(S)$, hence
\begin{equation*}
\theta(p) n^2\heightE{P}\leq\frac{3}{2}\chi(S)\cdot(d^{(p)}(n)-1)+(p^e-1)\chi(S)
\end{equation*}
and again by Proposition \ref{prop:arithm_estimates} it follows that
\begin{equation*}
\theta(p) n^2\heightE{P}\leq\frac{3}{2}\chi(S)\cdot\left(\frac{p^{e+1}-1}{(e+1)(p-1)}\cdot d(n)-1\right)+(p^e-1)\chi(S)
\end{equation*}
We rearrange the sum and drop several terms to get
\begin{equation*}
\theta(p) n^2\heightE{P}\leq\left(\frac{3}{2} p d(n)+1\right)p^e\chi(S)
\end{equation*}
For $\chi(S)=h_{K(C)}(E)\geq 2\cdot p^r (g(C)-1)$ the inequality
\[\frac{\chi(S)}{\heightE{P}}\leq 10^{18 p^{r}}\]
holds. Hence
\begin{equation*}
\theta(p) n^2\leq \left(\frac{3}{2} p d(n)+1\right) p^e\cdot 10^{18 p^{r}}
\end{equation*}
For $n\geq 19$ we obtain $d(n)\leq n^{\epsilon}$ with $\epsilon=0.988$. Since $p\geq 5$, then $\theta(p)\geq 2 - \frac{\pi^2}{5}>0.026$. We have $n=p^{e}n_{0}$ where $n_{0}$ is coprime to $p$. Finally
\[0.026\cdot n\cdot n_{0}\leq 10^{18p^r}\left(\frac{3}{2}p n^{\epsilon}+1\right).\]
We have $n_{0}\geq 1$ hence $\alpha n\leq \beta n^{\epsilon}+\gamma$ for explicit $\alpha,\beta$ and $\gamma$ that depend on $p$ and $r$ only. Such an inequality can hold only for finitely many $n$. We conclude that there exists a constant $N=N(p,r)$ such that for $n\geq N$ the divisor $D_{nP}$ has a primitive valuation.

For $\chi(S)=h_{K(C)}(E)< 2\cdot p^r (g(C)-1)$ the inequality
\[\frac{\chi(S)}{\heightE{P}}\leq 10^{36 g(C)p^{r}}.\]
holds. In a similar way as above we obtain a bound $N=N(g(C),p,r)$ such that for $n\geq N$ the divisor $D_{nP}$ has a primitive valuation.
\end{proof}
\begin{remark}
We observe that our leading assumption $p>3$ is needed to get a positive lower bound on $\theta(p)$. We leave it as an open question whether it is possible to establish the general result that will incorporate prime characteristics $2$ and $3$.
\end{remark}
Let us assume that $E$ is defined over $K(C)$ where the field of constants $K$ of $K(C)$ is not algebraically closed. For $\charac K(C)=p$ we put $K=\mathbb{F}_{q}$ where $q=p^{s}$ for some positive $s$. We consider a point $P$ in $E(K(C))$. It is possible to construct the fibration $\pi:S\rightarrow C$ such that the generic fibre is $E$ over $K(C)$ and the fibres above $v\in C(\overline{K})$ are defined over the field $k(v)$ which has $\deg v:=[k(v):K]$. 

\begin{theorem}\label{thm:theorem_wild_case}
Let $E$ be an elliptic curve defined over $K(C)$ of characteristic $p>3$ with field of constants $k=\mathbb{F}_{q}$, $q=p^{s}$. Let $E$ be wild. Let $\pi:S\rightarrow C$ be an elliptic fibration attached to $E$ in such a way that the fibres $E_{v}$ above $v\in C(\overline{K})$ of good reduction are defined over $k(v)$.  Take a point $P$ in $E(K(C))$ of infinite order. Let $p^{r}$ be the inseparable degree of the $j$--map of $E$. There exists an explicit constant $N=N(g(C),\chi(S),p,r,s)$ which depends only on the genus of $C$, Euler characteristic $\chi(S)$, $p$, $r$ and $s$ such that for $n\geq N$ the divisor $D_{nP}$ has a primitive valuation.
\end{theorem}
\begin{proof}
We proceed in a similar way to the proof of Theorem \ref{thm:main_theorem_tame}. Let $n$ be an integer such that the divisor $D_{nP}$ has no primitive valuation. For  $p\nmid n$ we follow the reasoning from the proof of Theorem \ref{thm:main_theorem_tame}. For $p\mid n$ , let $e=v_{p}(n)$. We arrive at the inequality
\begin{equation*}
\theta(p) n^2\heightE{P}\leq\frac{3}{2}\chi(S)\cdot(p^{e+1}\cdot d(n))+\deg W(E,P,n)
\end{equation*}
where $\theta(p)$ is defined as in the proof of Theorem \ref{thm:main_theorem_tame}.
For $v\in C(\overline{K})$ of good reduction the fibre $E_{v}$ is defined over $\mathbb{F}_{q^{\deg v}}$ and the reduction $P_{v}$ of point $P$ at $v$ is an $\mathbb{F}_{q^{\deg v}}$--rational point. From Lemma \ref{lem:Valutions_of_h_E_v} it follows that $\deg v \leq (p-1)\chi(S)$. Hasse--Weil bound \cite[Chap. V, Thm. 1.1]{Silverman_arithmetic} implies that
\[\# E_{v}(\mathbb{F}_{q^{\deg v}})\leq (\sqrt{q^{\deg v}}+1)^2.\]
From the definition of $m(v)$ we have $m(v)=\ord P_{v}$, hence 
\[m(v)\leq (\sqrt{q^{\deg v}}+1)^2\leq (\sqrt{q^{(p-1)\chi(S)}}+1)^2.\]
Let $k=\lceil\log_{p} (\frac{p+(p-1)^2\chi(S)}{2p-1})\rceil$ and suppose $e\leq k$. From Lemma \ref{lemma:wild_case} it follows that
\begin{align*}
\deg W(E,P,n)&\leq (p^{e}-1)\chi(S)\\
+\chi(S)&p^{2e}\heightE{P}\max\left\{144p^2,(\sqrt{q^{(p-1)\chi(S)}}+1)^4 \right\}+\frac{1}{2}\chi(S)^2.
\end{align*}
We conclude that there exist explicit constants $\alpha,\beta$ and $\gamma$ that depend on $\chi(S),p$ and $s$ such that
\[\theta(p)n^2\leq \frac{\chi(S)}{\heightE{P}}(\alpha d(n)+\beta)+\gamma.\]
We bound trivially $d(n)$ by $n$ from above. When we have $\chi(S)\geq 2\cdot p^r (g(C)-1)$ the bound $\frac{\chi(S)}{\heightE{P}}\leq 10^{18 p^{r}}$ holds and the inequality is true only for finitely many $n$ under the assumption $p\geq 5$. There is an explicit constant $N$ which depends on $\chi(S),p,s$ and $r$ such that for $n\geq N$ the divisor $D_{nP}$ has a primitive valuation. For $\chi(S)< 2\cdot p^r (g(C)-1)$ we produce a constant $N$ that depends additionally on $g(C)$. 

Finally, for $e>k$ we find explicit constants $\alpha,\beta,\gamma$ that depend on $\chi(S),p$ and $s$ such that
\[\theta(p) n^2\leq \frac{\chi(S)}{\heightE{P}}(\alpha d(n)+\beta)p^{e}+\gamma p^{e}.\]

For $n\geq 19$ we have $d(n)\leq n^{\epsilon}$ with $\epsilon=0.988$. Now we proceed as in the proof of Theorem \ref{thm:main_theorem_tame}.
\end{proof}

\section{Examples}\label{sec:Examples}
We present several examples where we establish the exact set of non--primitive divisors for concrete elliptic divisibility sequences. The first example deals with an infinite family of curves in characteristic $0$. We prove that as follows from the theorem the constant is absolute and in this case equals $1$, i.e. all divisors are primitive. 

The second example deals with the curve in characteristic $p=7$ where the $j$--map is inseparable. The next three examples indicate what happens when the field $K(C)$ is of positive characteristic and we allow the function $H(E)$ to vanish. We show that there are infinitely many non--primitive divisors in a sequence. They all rely on the fact that the multiplication by $p$ map is inseparable of degree $p^2$.

\begin{example}\label{example:myfamily}
We present now an example where the constant can be explicitly determined for a large family of elliptic curves with base curve $C=\mathbb{P}^{1}$ and $\chi(S)$ unbounded. The computations performed in this example inspired the proof of the general case for characteristic $0$ fields.

Computations in the example are based on \cite{Naskrecki_quadrics}.
Let $f,g,h\in\overline{\mathbb{Q}}[t]$ be polynomials of positive degree without a common root that satisfy $f^2+g^2=h^2$. We define an elliptic curve
\[E_{f,g,h}:y^2=x(x-f^2)(x-g^2)\]
over the function field $\overline{\mathbb{Q}}(t)$. There exists a point $Q=(-g^2,\sqrt{-2}g^2h)$ of infinite order on this curve. In the example we present an explicit argument that for all $n\in\mathbb{N}$ the divisors $D_{nQ}$ are primitive. Note that $\chi(S)=\deg f$ if $\deg g\leq \deg f$ so the Euler characteristic can be made unbounded. We can take for example polynomials \[(f,g,h)=\left(\frac{t^{2m}-1}{2},t^{m},\frac{t^{m}+1}{2}\right)\]
for any $m\in\mathbb{N}$. The equation $E_{f,g,h}$ represents the globally minimal Weierstrass model of the given elliptic curve. Its fibres of bad reduction are above the points $a\in\overline{\mathbb{Q}}$ such that $f(a)=0$ or $g(a)=0$ or $(f^2-g^2)(a)=0$ or $a=\infty$. The correcting terms in the Shioda's height formula are recorded in Table \ref{arr:corr_terms}. We denote by $v_{a}(\eta)$ the order of vanishing of a polynomial $\eta$ at $a$. We also denote $c_{v}(R,R)$ by $c_{v}(R)$. The height $\langle Q,Q\rangle$ equals $\deg f$. By the bilinearity of the height pairing $\langle\cdot,\cdot\rangle$ we know that $\langle k Q, k Q\rangle = k^2\langle Q,Q\rangle$. Application of \eqref{eq:Shioda_height} implies that
\[k^2\langle Q,Q\rangle = 2\deg f+2 \overline{kQ}.\overline{O}-\sum_{\substack{a:\\g(a)=0}} c_{a}(kQ) - c_{\infty}(kQ).\]
For $k$ even the sum $\sum\limits_{\substack{a:\\g(a)=0}}c_{a}(kQ)$ vanishes and for $k$ odd is equal to $\deg g$. Similarly for $2\mid k$ the factor $c_{\infty}(kQ)$ equals $0$ and for $2\nmid k$ it is equal to $\deg f-\deg g$. This follows from the group structure of $G(F_{v})$ for the fibres under consideration. By a simple algebraic manipulation we get the formula for the intersection numbers
\begin{equation*}
\overline{kQ}.\overline{O}=\left\{\begin{array}{ll}
\frac{k^2-2}{2}\deg f & ,2\mid k\\
& \\
\frac{k^2-1}{2}\deg f & ,2\nmid k
\end{array}\right.
\end{equation*}
Now we compute explicitly the constant $N(E_{f,g,h},Q)$. Suppose that $D_{nQ}$ does not have a primitive divisors. Then it follows
\[\dotprod{nP}{O}\leq \summy \dotprod{mP}{O}\]
Suppose $n$ is odd, then
\[\frac{n^2-1}{2}\deg f \leq \sum_{\substack{m\mid n\\ m< n}}\frac{m^2-1}{2}\deg f.\]
This is equivalent to
\begin{equation}\label{eq:sigma_bound}
(d(n)-1)+(n^2-1)\leq \sigma_{2}(n)-n^2.
\end{equation}
The first term on the left side of equation \eqref{eq:sigma_bound} is non--negative and $\sigma_{2}(n)<\zeta(2)n^2$, so
\[n^2 < \frac{1}{2-\zeta(2)}\]
hence $n< 1.68$, so $n=1$. Now we consider the case when $n$ is even. The inequality
\[\frac{n^2-2}{2}\deg f \leq \sum_{\substack{m\mid n\\ m<n}} \overline{m Q}.\overline{O}\]
is equivalent to
\[\left(2d(n)-d\left(n/2^{v_{2}(n)}\right)\right)+2(n^2-2)\leq \sigma_{2}(n).\]
We drop the non-negative term $\left(2d(n)-d\left(n/2^{v_{2}(n)}\right)\right)$. It follows that
\[(2-\zeta(2))n^2\leq  4\]
which can hold only for $n\leq 2$. Now we check by a direct computation that $D_{2Q}$ actually contains primitive valuations:
\[2Q=\left(-\frac{(f^2-g^2)^2}{8h^2}, \frac{\sqrt{-1}(g^2-f^2)(3f^2+g^2)(f^2+3g^2)}{16\sqrt{2}h^3} \right)\]
so the constant $N(E_{f,g,h},Q)$ equals $1$.

\begin{table}
$\begin{array}{cccc}
v & c_{v}(Q) & \textrm{comp}_{v}(Q) & F_{v} \\
\hline
\infty & \deg f - \deg g & 2(\deg f - \deg g) & I_{4(\deg f - \deg g)}\\
a: g(a)=0 & v_{a}(g) & 2 v_{a}(g) & I_{4v_{a}(g)}\\
a: f(a)=0 & 0 & 0 & I_{4v_{a}(g)}\\
a: (f^2-g^2)(a)=0 & 0 & 0 & I_{2v_{a}(g)}
\end{array}$
\vskip5pt
\caption{Correcting terms for a curve with Weierstrass equation $E_{f,g,h}$}\label{arr:corr_terms}
\end{table}
\end{example}
 
\begin{example}
Let $C=\mathbb{P}^{1}$ with parameter $t$ for its function field $K(C)$. Assume that $K=\overline{\mathbb{F}}_{7}$. The curve $E: y^2=x^3-t^3 x+t$ has bad reduction at $t=0$ (type $II$), $t=5$ (type $I_{7}$) and $t=\infty$ (type $III$). The associated elliptic surface $\pi: S\rightarrow C$ satisfies $\chi(S)=1$ and hence $S$ is a rational surface with Picard number equal to $10$. By Shioda--Tate formula \cite[Cor. 5.3]{Shioda_Mordell_Weil} the group $E(\overline{\mathbb{F}}_{7}(t))$ has rank $1$ and  by \cite{Oguiso_Shioda} is generated by $P=(3t+2,2t^2+t+1)$ which has canonical height $\heightE{P}=\frac{1}{2}\langle P,P\rangle = \frac{1}{14}$. The four points $P,2P,3P$ and $4P$ are integral with respect to $t$.
\begin{table}[h!]
\[\begin{array}{c|c}
  n & D_{nP}\\
  \hline
  1 & 0\\
  2 & 0\\
  3 & 0\\
  4 & 0\\
  5 & (4)\\
  6 & (3)\\
  7 & (0)\\
  8 & (\alpha_{1})+(\alpha_{2}) ((t-\alpha_{1})(t-\alpha_{2}) = t^2+6t+4\\
  \ldots & \ldots\\
  14 & (0)+5(\infty)
 \end{array}\]
 \caption{Divisors $D_{nP}$ for small values of $n$}
\end{table}
We prove below that these are the only integral points with respect to $t$ and for all $n\geq 5$ the divisor $D_{nP}$ admits a primitive valuation. Observe that the $j$--invariant of $E$ is a $7$-th power $j=\left(\frac{6t}{t+2}\right)^7$ and its inseparable degree is $7$.

\begin{table}\label{tab:bad_val_example_char_7}
 \[\begin{array}{c|c|c|c|c|}
  v & \textrm{type of } v &  P_{v} & \textrm{Is singular on }E_{v}\textrm{?} & c_{v}(P,P)\\
  \hline
  t=5 & I_{7} & (3,0) & \textrm{yes} & 10/7\\
  t=0 & II & (2,1) & \textrm{no} & 0\\
  t=\infty & III & (0,0) & \textrm{yes} & 1/2
 \end{array}\]
\caption{Reduction $P_{v}$ of point $P$ at places $v$ of bad reduction with reduced curve $E_{v}$}
\end{table}
We check that for $v\neq 0, \infty$ we have $h_{E,v}=0$ and $h_{E,0}=1$, $h_{E,\infty}=5$, and $m(0)=7$, $m(\infty)=14$. Information from Table \ref{tab:bad_val_example_char_7} and knowledge of the component group for each bad fibre allow us to compute
\begin{equation*}
 c_{v}(kP,kP)=\left\{\begin{array}{ll}
                      1/2 & , v=\infty, 2\nmid k\\
                      \frac{(2k\textrm{ mod }7)\cdot (7-(2k\textrm{ mod }7))}{7} & , v=5\\
                      0 & \textrm{, otherwise}
                      \end{array}
 \right.
\end{equation*}
We assume $n>1$ and that $D_{nP}$ has no primitive divisors. From the formula
\begin{equation}\label{eq:divisor_ineq_example_char_7}
D_{nP}\leq\sum_{\substack{m\mid n\\ m<n}}p^{v_{p}(\frac{n}{m})}D_{mP}+ W(E,P,n) 
\end{equation}
and the computations above we can effectively check that for $5\leq n\leq 20000$ the formula does not hold. For $n\geq 20000$ we apply the degree function to \eqref{eq:divisor_ineq_example_char_7} and get
\[0.12n\leq 14\cdot \left(\frac{31}{24}\cdot n^{0.465}+1\right)\]
which is valid only for $n\leq 11998$. So only the divisors $D_{2P},D_{3P}$ and $D_{4P}$ are not primitive and for $n\geq 5$ the divisor $D_{nP}$ always has a primitive valuation. Because $E(\overline{\mathbb{F}}_{7}(t))=\langle P\rangle $, so for any $\overline{\mathbb{F}}_{7}(t)$--rational point $Q$ on $E$ the sequence $D_{nQ}$ contains at most $3$ non--primitive elements.
\end{example}

\begin{example}
Let $p\geq 5$ and pick an elliptic curve $E_{0}:y^2=x^3+\alpha x+\beta$ with $\alpha,\beta\in \mathbb{F}_{p}$ which is supersingular. Consider the field $K(C)=\mathbb{F}_{p}(t)$ of functions of the projective line $C$ over $\mathbb{F}_{p}$ and let $r=t^3+\alpha t+\beta$. The curve $E_{0}^{(r)}: y^2=x^3+\alpha r^2 x+\beta r^3$ over $K(C)$ is a generic fibre of a Kummer K3 surface with $I_{0}^{*}$ fibres at places $t_{0}$ such that $r(t_{0})=0$ or $t_{0}=\infty$.
We always have a point $P=(tr,r^2)$ on this curve
(in fact $\rank E_{0}^{(r)}(\overline{\mathbb{F}}_{p}(t))=4$ because $E_{0}$ is supersingular, cf.\cite[\S 12.7]{Shioda_Schutt}). 
Moreover on $E_{0}$ the $[p]$ multiplication map is inseparable of degree $p^{2}$ and since $E_{0}$ is defined over $\mathbb{F}_{p}$ we have that $[p](x,y)=(x^{p^2},-y^{p^2})$. The curve $E_{0}$ over $\overline{K(C)}$ is isomorphic to $E_{0}^{(r^{d})}$ over $\overline{K(C)}$ via $(x,y)\mapsto (x r^{d},y^{3/2 d})$ for any positive integer $d$. Hence the $[p]$ map on $E_{0}^{(r)}$ satisfies $[p](x,y)=(x^{p^2} r^{1-p^2}, -y^{p^2} r^{(3-3p^2)/2})$. Any $p^{k}$ multiple of the point $P$ on $E_{0}^{(r)}$ is an integral point
\[p^{k}P=(t^{p^{2k}}r,r^{(3+p^{2k})/2}).\]
The sequence $\{D_{p^{k}P}\}_{k\geq 0}$ of divisors has support only at $t=\infty$: $D_{P}=0$ and $D_{p^{k}P}=(p^2-1)(\infty)$ for $k\geq 1$. Hence the sequence $\{D_{nP}\}_{n\geq 1}$ has infinitely many elements that have no primitive valuation. 

There is nothing special about the point $P$ so we can pick any $K(C)$--rational point $Q$ on $E_{0}^{(r)}$ and there will exist again infinitely many divisors $D_{nQ}$ for $n\geq 1$. From our construction it follows that $H(E)=0\in K(C)$.
\end{example}

\begin{example}
Let $E$ be an elliptic curve over $\mathbb{F}_{2}(t)$ with globally minimal Weierstrass equation
\[E:y^2+ty=x^3+x.\]
We consider the point $P=(1,0)$ which is of infinite order in $E(\mathbb{F}_{2}(t))$. Multiplication by $2$ map on $E$ satisfies the equality
\[x([2](x,y))=\frac{1 + x^4}{t^2}.\]
For two polynomials $p$,$s$ in $\mathbb{F}_{2}[t]$ which are coprime and $p/s^2$ is the $x$--coordinate a point $Q$ on $E$ we get
\[x(2Q)=\frac{p^4+s^8}{t^2 s^8}\]
and it is easy to see that $p^4+s^8$ and $t^2 s^8$ are again coprime. We show by induction that for $l\geq 1$
\[x(2^l P)= \frac{\sum\limits_{j=1}^{l-1}t^{\sum_{k=j}^{l-2} 2^{2k+1}}}{t^{\frac{2}{3}(2^{2l-2}-1) }}.\]
So for $l\geq 2$ we have $\Supp D_{2^l P}=\{(0)\}$ and for every $l\ge 3$ the divisor $D_{2^l P}$ is not primitive.
\end{example}

\begin{example}
Let $E$ be an elliptic curve over $\mathbb{F}_{3}(t)$ with globally minimal Weierstrass equation
\[E:y^2+txy=x^3+2t^2x^2+(2t^2+1)x+(2t^2+1).\]
The point $P=(1,0)$ is of infinite order in $E(\mathbb{F}_{3}(t))$. We check that
\[x([3](x,y)) = \frac{1}{(1+t)^4 (2+t)^4} x^9 +\frac{2t^2}{(1+t)(2+t)}.\]
For $l\geq 1$ the divisor $D_{3^l P}$ is supported at $1$ and $2$ and for $l\geq 2$ it is not primitive.
\end{example}

\section*{Acknowledgments}
The author would like to thank Wojciech Gajda for suggesting this research problem and for the hint to the height bounds of \cite{SilvHindry} used at the critical point of the argument. He also thanks Krzysztof Górnisiewicz for helpful remarks, Maciej Radziejewski for his information about the upper bounds on the divisor sum functions and Joseph Silverman for helpful comments. Finally the author thanks an anonymous referee for careful reading of the manuscript and for very helpful suggestions.

\bibliography{bibliography}
\bibliographystyle{amsplain}
\end{document}